\documentclass[11pt]{article}
\usepackage[T1]{fontenc}
\usepackage[utf8]{inputenc}
\usepackage[margin=1.08in]{geometry}
\usepackage{amsmath,amssymb,amsthm,mathtools}
\usepackage{enumitem}
\usepackage{xcolor}
\usepackage[colorlinks=true,linkcolor=blue,citecolor=blue,urlcolor=blue]{hyperref}
\usepackage[nameinlink,noabbrev,capitalize]{cleveref}
\usepackage[symbol]{footmisc}
\numberwithin{equation}{section}
\newtheorem{theorem}{Theorem}[section]
\newtheorem{proposition}[theorem]{Proposition}
\newtheorem{lemma}[theorem]{Lemma}

\theoremstyle{remark}

\newcommand{\Z}{\mathbb Z}
\newcommand{\td}{\tilde{d}}
\newcommand{\te}{\tilde{e}}

\title{On zero-sum Ramsey numbers of complete bipartite graphs}
\author{Cheng Chi$^\ast$ 
\and Jialin He$^{\dagger,\S}$}
\date{}

\begin{document}
\maketitle
\footnotetext[1]{School of Mathematical Sciences, Shanghai Jiao Tong University, 800 Dongchuan Road, Shanghai 200240, China.
    Email: chengchi@sjtu.edu.cn. Supported by National Key R\&D Program of China under grant No. 2022YFA1006400 and National Natural Science Foundation of China No. 12571376.}

\footnotetext[2]{Corresponding Author.}

\footnotetext[4]{School of Mathematical Sciences,  Key Laboratory of MEA (Ministry of Education) and Shanghai Key Laboratory of PMMP,  East China Normal University, Shanghai 200241, China.
    Email: jlhe@math.ecnu.edu.cn. Supported in part by Science and Technology Commission of Shanghai Municipality No. 22DZ2229014.}

\begin{abstract}
    For an integer $q\ge 2$ and a graph $F$ satisfying $q\mid e(F)$, the zero-sum Ramsey number $R(F,\Z_q)$ is the least integer $n$ such that every edge-labeling $w\colon E(K_n)\to \Z_q$ contains a copy of $F$ whose edge-label sum is zero in $\Z_q$.
    Write $K_{s,t}$ for the complete bipartite graph with $s$ vertices on one side and $t$ vertices on the other side.

    We prove that for every integer $q\ge2$, there is an explicit threshold $S(q)$ such that $R(K_{s,qk},\Z_q)=s+qk$ for all $s\ge S(q)$ and all $k\ge1$.
    We also determine the zero-sum Ramsey number of $K_{s,3k}$ over $\Z_3$ completely. 
    We show that $R(K_{s,3k},\Z_3)=s+3k$, except when $s=2$ and $k\ge1$, or when $s\in\{3,4,5,7\}$ and $k=1$.
    In these exceptional cases, $R(K_{s,3k},\Z_3)=s+3k+1$.
    In particular, this implies that the threshold $S(q)$ is best possible for \(q=3\).
\end{abstract}

\section{Introduction}

Zero-sum Ramsey theory is a group-labelled analogue of classical Ramsey theory.
Instead of seeking a monochromatic copy of a prescribed graph, one labels the edges of a complete graph by elements of a finite abelian group and asks for a copy whose total edge-label is the identity.
This area was initiated by Bialostocki and Dierker~\cite{BD1990,BD1992}, motivated in part by the Erd\H{o}s--Ginzburg--Ziv theorem~\cite{EGZ1961} and the Cauchy--Davenport theorem~\cite{Cauchy1813,Davenport1935}; see Caro's survey~\cite{Caro1996} for background.

Throughout the paper we work with the cyclic group \(\Z_q\), where \(q\ge2\).
For a graph \(G\), the \emph{zero-sum Ramsey number} \(R(G,\Z_q)\) is the least integer \(n\) such that every edge-labeling $w\colon E(K_n)\to \Z_q$ contains a copy \(G'\) of \(G\) with $\sum_{e\in E(G')}w(e)=0$ in $\Z_q$.
The divisibility condition $q\mid e(G)$ is necessary, since the constant labeling $w(e)=1$ gives a nonzero sum on every copy of $G$.
Conversely, if \(q\mid e(G)\), then \(R(G,\Z_q)\) is finite by the ordinary multicolor Ramsey theorem~\cite{Ramsey1930}, because any monochromatic copy of \(G\) is zero-sum. 
A related bipartite-host version was studied earlier by Caro~\cite{Caro1993Bipartite}. In that setting, one colors the edges of a complete bipartite host graph and asks for a zero-sum copy of a prescribed bipartite graph. This differs from the present parameter, where the host graph is complete and the target graph is \(K_{s,t}\).

Determining exact values of \(R(G,\Z_q)\) is difficult beyond a few special cases.
For $q=2$, Caro gave a complete
characterization of \(R(G,\Z_2)\) for every graph \(G\) with an even number of edges~\cite{Caro1994}.
For \(q\ge3\), no comparable general theory is known.
Even for complete graphs, substantial work of Alon and Caro~\cite{AlonCaro1993} , Caro~\cite{Caro1992Complete,Caro1997}, and Harborth and Piepmeyer~\cite{HarborthPiepmeyer1994,HarborthPiepmeyer1996} leaves some cases open, including residue classes over \(\Z_3\); see~\cite{CaroMifsud2025}.
Other exact results are known for stars and matchings~\cite{BD1992,Caro1992Stars}. 
Trees and forests have also received considerable attention.
Caro and Roditty posed a conjecture concerning the zero-sum Ramsey number of trees~\cite{CaroRoditty1995}.
More recently, Caro and Mifsud initiated a systematic study over \(\Z_3\), proving sharp \(n+O(1)\)-type bounds for forests~\cite{CaroMifsud2025}.
Subsequently, Alvarado, Colucci and Parente determined the exact \(\Z_3\)-Ramsey numbers of forests without isolated vertices~\cite{AlvaradoColucciParente2025}.

For graphs containing cycles, much less is known.
Cycles are the most basic family in this direction.
Related zero-sum cycle-existence results have been obtained for group-labelled complete digraphs~\cite{LetzterMorrison2024}.
For ordinary cycles of prescribed length, several ranges of zero-sum Ramsey numbers are known, but the general picture remains incomplete.
For example, Caro~\cite{Caro1992Complete,Caro1996} conjectured that \(R(C_q,\Z_q)=4q-3\) for every odd \(q\ge5\), a problem that remains open in general.
See~\cite{ChiHeCyclesWheels2026,HeathSimmons2026} for recent progress on cycles and wheels.

Recent general results show that zero-sum Ramsey numbers are often close to the vertex-count lower bound for sparse or locally sparse graphs. Katz, Lian, Malekshahian and Shapiro proved linear bounds for bounded-degree graphs over finite abelian groups~\cite{KatzLianMalekshahianShapiro2025}; for prime moduli, sharper \(n+O_p(1)\)-type bounds have been obtained for forests~\cite{ColucciDEmidio2025}, degenerate graphs under suitable hypotheses~\cite{Shapiro2026}, and graphs with large \(2\)-packings~\cite{HeathSimmons2026}. Complete bipartite graphs \(K_{s,t}\), however, fall outside these frameworks when both parts grow: they have unbounded maximum degree and degeneracy, diameter two, and no large \(2\)-packing. Their natural lower bound is simply
\(R(K_{s,t},\mathbb Z_q)\ge s+t\).

Our first contribution is to show that for every integer $q\ge2$, this trivial lower bound is sharp for complete bipartite graphs, once \(s\) is sufficiently large as a function of \(q\).  
For an integer $q\ge2$, let $g=\gcd(q,2)$, $d=q/g$ and $M(q)=R(K_q,\Z_d)$.
When \(d=1\), we regard \(\Z_1\) as the trivial group and set \(R(K_m,\Z_1)=m\).
Let $S(q)=gM(q)-q-g+1$.

\begin{theorem}\label{thm:eventual-general}
    Let $q\ge2$ be an integer.
    For every $k\ge1$ and every $s\ge S(q)$, we have
    \[
        R(K_{s,qk},\Z_q)=s+qk.
    \]
\end{theorem}
Theorem~\ref{thm:eventual-general} may be compared with the following
elementary upper bound due to Erd\H{o}s--Ginzburg--Ziv theorem~\cite{EGZ1961}. For all integers \(q\ge2\), \(s\ge1\), and \(k\ge1\),
\[
    R(K_{s,qk},\Z_q)\le s+qk+q-1 .
\]
Indeed, let \(N=s+qk+q-1\), and let
\(w\colon E(K_N)\to\Z_q\) be an arbitrary edge-labeling. Fix an \(s\)-set
\(A=\{u_1,\ldots,u_s\}\subseteq V(K_N)\), and let
\(C=V(K_N)\setminus A\). Then \(|C|=qk+q-1\). For each \(v\in C\), define
\[
    a_v=\sum_{u\in A} w(uv)\in \Z_q .
\]
By the Erd\H{o}s--Ginzburg--Ziv theorem, the sequence
\((a_v)_{v\in C}\) contains a subsequence indexed by a set
\(B\subseteq C\) with \(|B|=qk\) and $\sum_{v\in B} a_v=0$ in $\Z_q$.  
Hence, $\sum_{v\in B}\sum_{u\in A} w(uv)=0$ in $\Z_q$. 
Thus the complete bipartite subgraph with parts \(A\) and \(B\) is a zero-sum copy of \(K_{s,qk}\). This gives a general upper bound which is only \(q-1\)
above the vertex-count lower bound \(s+qk\). Theorem~\ref{thm:eventual-general}
shows that, once \(s\ge S(q)\), these extra \(q-1\) vertices are unnecessary.

Particularly, we study the case $q=3$.
Our second contribution is a complete determination of $R(K_{s,3k},\Z_3)$, including all exceptional small cases.
The proof uses a lifting method different from the one used to prove \Cref{thm:eventual-general}.

\begin{theorem}\label{thm:z3-exact}\label{thm:main}
    For all integers $s\ge2$ and $k\ge1$,
    \[
        R(K_{s,3k},\Z_3)=
        \begin{cases}
            s+3k+1, & s=2,\ k\ge1,\ \text{or}\ s\in\{3,4,5,7\},\ k=1, \\
            s+3k,   & \text{otherwise}.
        \end{cases}
    \]
\end{theorem}
For \(q=3\), the EGZ bound above gives
\(R(K_{s,3k},\mathbb Z_3)\le s+3k+2\). Hence
Theorem~\ref{thm:z3-exact} improves this general bound by two vertices in all non-exceptional cases and by one vertex in the exceptional cases.
Moreover, the threshold $S(q)$ in Theorem~\ref{thm:eventual-general} is best possible for $q=3$.
Indeed, Lemma~\ref{lem:cg} gives \(M(3)=R(K_3,\Z_3)=11\), and hence \(S(3)=8\).
Moreover, \Cref{thm:main} gives \(R(K_{7,3},\Z_3)=11>10=7+3\), while \(R(K_{s,3k},\Z_3)=s+3k\) for all \(s\ge8\) and \(k\ge1\).
Thus the condition \(s\ge S(3)\) cannot be improved.

The rest of the paper is organized as follows.
\Cref{sec:defs-lemmas} contains the notation and some general auxiliary lemmas.
\Cref{sec:eventual-proof} proves \Cref{thm:eventual-general}.
The proof of \Cref{thm:z3-exact} begins in \Cref{sec:main-proof} and continues with the proofs of the base lemmas in \Cref{Sec:pf lemmas}.

\section{Definitions and lemmas}\label{sec:defs-lemmas}

Throughout the paper, all edge-label sums are taken in the relevant cyclic group.
For a subgraph $H$ of a labeled complete graph, write $w(H)=\sum_{e\in E(H)}w(e)$.
The subgraph $H$ is \emph{zero-sum} if $w(H)=0$.
If $A$ and $B$ are disjoint vertex sets, define
\[
    w(A,B)=\sum_{a\in A}\sum_{b\in B}w(ab).
\]
Thus a copy of $K_{s,m}$ with sides $A$ and $B$ is zero-sum if $|A|=s$, $|B|=m$, and $w(A,B)=0$.

Let $K$ be a complete graph on vertex set $V$ with edge-labeling $w$.
For a vertex $x\in V$, define its \emph{weighted degree} by $\td_w(x)\coloneqq\sum_{y\ne x}w(xy)$.
For $X\subseteq V$, define
\[
    \te_w(X)\coloneqq\sum_{\{x,y\}\subseteq X}w(xy)\text{ and }\partial_w(X)\coloneqq w(X,V\setminus X).
\]
When the context is clear, we write $\td(x)$, $\te(X)$ and $\partial(X)$.

For the proof of \Cref{thm:z3-exact}, we shall use the following result on edge-colored complete graphs proved by Chung and Graham~\cite{ChungGraham1983}.

\begin{lemma}[Chung and Graham, Theorem 1 in~\cite{ChungGraham1983}]\label{lem:cg}
    We have \(R(K_3,\Z_3)=11\). Equivalently, every edge-coloring of \(K_{11}\) with three colors contains a triangle which is not exactly two-colored, while there exists a three-coloring of \(K_{10}\) in which every triangle is exactly two-colored.
\end{lemma}

Next, we state a lifting lemma showing that $R(K_{s+t,3k},\Z_3)$ is controlled by $R(K_{s,3k},\Z_3)+t$.

\begin{lemma}\label{lem:side-lift}
    Let $s\ge1$, $k\ge1$, and $t\ge0$.  If $3\nmid s$, then $R(K_{s+t,3k},\Z_3)\le R(K_{s,3k},\Z_3)+t$.
\end{lemma}
\begin{proof}
    Set $n=R(K_{s,3k},\Z_3)+t$. Let $K$ be a complete graph on vertex set $V$ with $|V|=n$, and let $w\colon E(K)\to\Z_3$ be an arbitrary edge-labeling.  Choose a $t$-set $Z\subseteq V$ and let $X=V\setminus Z$.
    For $x\in X$, define $h(x)=\sum_{z\in Z}w(zx)$.
    Define an auxiliary edge-labeling $w'$ on the pairs from $X$ by
    \[
        w'(xy)=w(xy)+s(h(x)+h(y)).
    \]
    Since $|X|=R(K_{s,3k},\Z_3)$, there are disjoint sets $A,B\subseteq X$ with $|A|=s$, $|B|=3k$, and $w'(A,B)=0$.  Expanding this equality gives
    \begin{equation}\label{eq:side-lift-expand}
        0=w(A,B)+s |B|\sum_{a\in A}h(a)+s |A|\sum_{b\in B}h(b).
    \end{equation}
    Now $|B|=3k=0$ and $s|A|=s^2=1$ in $\Z_3$, since $3\nmid s$, so \eqref{eq:side-lift-expand} becomes $0=w(A,B)+\sum_{b\in B}h(b)$.  Since $\sum_{b\in B}h(b)=w(Z,B)$, we get $w(A\cup Z,B)=0$.  Thus $K$ contains a zero-sum $K_{s+t,3k}$, which implies that $R(K_{s+t,3k},\Z_3)\le n$, proving the desired result.
\end{proof}

\section{Proof of Theorem~\ref{thm:eventual-general}}\label{sec:eventual-proof}

The first step is an elementary cut-sum identity.
It measures how the cut sum changes when a set of new vertices is inserted into one side of a cut.

\begin{proposition}\label{prop:cut-increment}
    Let $K$ be a complete graph on vertex set $V$, and let $w\colon E(K)\to\Z_q$ be an edge-labeling.
    Let $U\subseteq V$, write $X=V\setminus U$, and for each $x\in X$ let
    \[
        b_x=\left(\sum_{y\in X\setminus\{x\}}w(xy)-\sum_{u\in U}w(ux)\right)\in\Z_q.
    \]
    Then, for every $Q\subseteq X$,
    \begin{equation}\label{eq:cut-increment}
        \partial(U\cup Q)-\partial(U)=\left(\sum_{x\in Q}b_x-2\sum_{\{x,y\}\subseteq Q}w(xy)\right)
        \text{ in }\Z_q.
    \end{equation}
\end{proposition}

\begin{proof}
    By the definition of the cut sum,
    \[
        \begin{aligned}
            \partial(U\cup Q)-\partial(U)
             & =\sum_{\substack{x\in Q \\ y\in X\setminus Q}}w(xy)
            -\sum_{\substack{u\in U    \\ x\in Q}}w(ux).
        \end{aligned}
    \]
    Also,
    \[
        \begin{aligned}
            \sum_{x\in Q}b_x
             & =\sum_{\substack{x\in Q \\ y\in X\setminus Q}}w(xy)
            +\sum_{\substack{x,y\in Q  \\ x\ne y}}w(xy)
            -\sum_{\substack{u\in U    \\ x\in Q}}w(ux)  \\
             & =\sum_{\substack{x\in Q \\ y\in X\setminus Q}}w(xy)
            +2\sum_{\{x,y\}\subseteq Q}w(xy)
            -\sum_{\substack{u\in U    \\ x\in Q}}w(ux).
        \end{aligned}
    \]
    Subtracting the internal-edge term gives \eqref{eq:cut-increment}.
\end{proof}

Using this proposition, we prove the following lemma.
\begin{lemma}\label{lem:block-insertion}
    Let $q\ge2$, let $g=\gcd(q,2)$, $d=q/g$, and set $M(q)=R(K_q,\Z_d)$.
    Let $K$ be a complete graph on vertex set $V$, and let $w\colon E(K)\to\Z_q$ be an edge-labeling.
    Suppose that $U\subseteq V$, set $X=V\setminus U$, and assume that $\partial(U)=0$.
    If $|X|\ge gM(q)-g+1$, then there exists a subset $Q\subseteq X$ such that $|Q|=q$ and $\partial(U\cup Q)=0$.
\end{lemma}

\begin{proof}
    The goal is to choose $q$ vertices $Q\subseteq X$ so that the cut increment in \eqref{eq:cut-increment} vanishes.
    We first find a subset on which the linear terms $b_x$ lie in a common residue class modulo $g$, and then apply the definition of $M(q)$ to an auxiliary $\Z_d$-edge-labeling.

    For each $x\in X$, let
    \[
        b_x=\sum_{y\in X\setminus\{x\}}w(xy)-\sum_{u\in U}w(ux)\in\Z_q.
    \]
    By Proposition~\ref{prop:cut-increment}, it suffices to find $Q\subseteq X$ with $|Q|=q$ such that
    \[
        \sum_{x\in Q}b_x-2\sum_{\{x,y\}\subseteq Q}w(xy)=0
        \text{ in }\Z_q.
    \]

    Partition \(X\) according to the residue of \(b_x\) modulo \(g\).
    There are \(g\) residue classes.
    Since $|X|\ge gM(q)-g+1=g(M(q)-1)+1$, one of these classes contains at least \(M(q)\) vertices.
    Thus there are a residue \(\epsilon\in\Z_g\) and a subset \(Y\subseteq X\) with \(|Y|\ge M(q)\) such that $b_x\equiv \epsilon\pmod g$ for every \(x\in Y\).
    We also write $\epsilon$ for its representative in $\{0,\ldots,g-1\}$.
    For each $x\in Y$, there is a unique $a_x\in\Z_d$ such that
    \begin{equation}\label{eq:b-decomposition}
        b_x=\epsilon+ga_x
        \text{ in }\Z_q.
    \end{equation}
    Here $ga_x$ denotes the image of $a_x\in\Z_d$ under the natural homomorphism $\Z_d\to\Z_q$, $a\mapsto ga$.

    Let $\overline w(xy)$ denote the reduction of $w(xy)$ modulo $d$, and let $c=2/g$.
    Thus $c=2$ when $q$ is odd and $c=1$ when $q$ is even.
    We define an auxiliary edge-labeling $\psi$ on the pairs from $Y$ as follows.
    If $d=1$, let $\psi\equiv0$.
    If $d>1$, then $q\equiv0\pmod d$, so $q-1\equiv -1\pmod d$ and $q-1$ is invertible in $\Z_d$; define
    \[
        \psi(xy)=-c\,\overline w(xy)+(q-1)^{-1}(a_x+a_y)
        \text{ in }\Z_d.
    \]
    For every $Q\subseteq Y$ with $|Q|=q$, we have
    \begin{equation}\label{eq:psi-sum}
        \sum_{\{x,y\}\subseteq Q}\psi(xy)=\sum_{x\in Q}a_x-c\sum_{\{x,y\}\subseteq Q}\overline w(xy)
        \text{ in }\Z_d.
    \end{equation}
    Indeed, for $d=1$, \eqref{eq:psi-sum} is trivial since $\Z_1$ contains the identity element only.
    For \(d>1\), fix \(x\in Q\). In the sum
    \(\sum_{\{u,v\}\subseteq Q}(a_u+a_v)\), the term \(a_x\) appears once for each edge of \(K_Q\) incident with \(x\). Since \(|Q|=q\), there are \(q-1\) such edges.
    Hence
    \[
        \sum_{\{x,y\}\subseteq Q}(a_x+a_y)=(q-1)\sum_{x\in Q}a_x .
    \]
    This implies
    \[
        \begin{aligned}
            \sum_{\{x,y\}\subseteq Q}\psi(xy)
             & =-c\sum_{\{x,y\}\subseteq Q}\overline w(xy)
            +(q-1)^{-1}\sum_{\{x,y\}\subseteq Q}(a_x+a_y)                   \\
             & =-c\sum_{\{x,y\}\subseteq Q}\overline w(xy)
            +(q-1)^{-1}(q-1)\sum_{x\in Q}a_x                                \\
             & =\sum_{x\in Q}a_x-c\sum_{\{x,y\}\subseteq Q}\overline w(xy).
        \end{aligned}
    \]

    Since $|Y|\ge M(q)=R(K_q,\Z_d)$, the edge-labeling $\psi$ contains a zero-sum copy of $K_q$.
    Thus there exists $Q\subseteq Y$ such that $|Q|=q$ and
    \begin{equation}\label{eq:zero-psi}
        \sum_{\{x,y\}\subseteq Q}\psi(xy)=0
        \text{ in }\Z_d.
    \end{equation}
    For this $Q$, combining \eqref{eq:psi-sum} and \eqref{eq:zero-psi} gives
    \[
        \sum_{x\in Q}a_x-c\sum_{\{x,y\}\subseteq Q}\overline w(xy)=0
        \text{ in }\Z_d.
    \]
    Multiplying by $g$ and using $gc=2$, we obtain
    \begin{equation}\label{eq:key-congruence}
        g\sum_{x\in Q}a_x-2\sum_{\{x,y\}\subseteq Q}w(xy)=0
        \text{ in }\Z_q.
    \end{equation}
    The replacement of $\overline w(xy)$ by $w(xy)$ in \eqref{eq:key-congruence} is valid because $\overline w(xy)$ is the reduction of $w(xy)$ modulo $d$, and $q=gd$ divides $2d$.

    By \eqref{eq:b-decomposition} and $|Q|=q$,
    \begin{equation}\label{eq:b-sum}
        \sum_{x\in Q}b_x=q\epsilon+g\sum_{x\in Q}a_x=g\sum_{x\in Q}a_x
        \text{ in }\Z_q.
    \end{equation}
    Equations \eqref{eq:key-congruence} and \eqref{eq:b-sum} give
    \[
        \sum_{x\in Q}b_x-2\sum_{\{x,y\}\subseteq Q}w(xy)=0
        \text{ in }\Z_q.
    \]
    Therefore \eqref{eq:cut-increment} yields $\partial(U\cup Q)-\partial(U)=0$.
    Since $\partial(U)=0$, we have $\partial(U\cup Q)=0$, as required.
\end{proof}

Now we are ready to prove \cref{thm:eventual-general}.

\medskip

The lower bound $R(K_{s,qk},\Z_q)\ge s+qk$ holds trivially since a complete graph on $s+qk-1$ vertices cannot contain a copy of $K_{s,qk}$.
It remains to prove the upper bound.
Let $K$ be a complete graph on vertex set $V$ with $|V|=s+qk$, and let $w\colon E(K)\to\Z_q$ be an arbitrary edge-labeling.

We construct nested sets $U_0\subset U_1\subset\cdots\subset U_k\subseteq V$ such that, for every $0\le i\le k$, $|U_i|=iq$ and $\partial(U_i)=0$.
Start with $U_0=\emptyset$.
Suppose that $U_i$ has been constructed for some $0\le i<k$, and write $X_i=V\setminus U_i$.
By the choice of $s$,
\[
    |X_i|=s+q(k-i)\ge s+q\ge gM(q)-g+1.
\]
Hence Lemma~\ref{lem:block-insertion} gives a set $Q_i\subseteq X_i$ with $|Q_i|=q$ and $\partial(U_i\cup Q_i)=0$.
Let $U_{i+1}=U_i\cup Q_i$.
Then $|U_{i+1}|=(i+1)q$ and $\partial(U_{i+1})=0$, so the induction continues.

After $k$ steps, $|U_k|=qk$ and $\partial(U_k)=0$.
Let $B=U_k$ and $A=V\setminus U_k$.
Then $|B|=qk$, $|A|=s$, and
\[
    w(A,B)=\partial(B)=\partial(U_k)=0.
\]
The complete bipartite subgraph with parts $A$ and $B$ is therefore a zero-sum copy of $K_{s,qk}$.
Since $w$ was arbitrary, every $\Z_q$-edge-labeling of $K_{s+qk}$ contains a zero-sum $K_{s,qk}$. This proves the desired upper bound and completes the proof of Theorem~\ref{thm:eventual-general}. \qed

\section{Proof of Theorem~\ref{thm:main}}\label{sec:main-proof}

In this section, we first isolate three base results and then use them to prove Theorem~\ref{thm:main}. The proofs of these results are deferred to \Cref{Sec:pf lemmas}.
In Subsection~\ref{subsec:small-lower}, we establish the lower bounds for the exceptional $K_{s,3}$ cases by constructing explicit edge-labelings.
The proof of Theorem~\ref{thm:main} is given in Subsection~\ref{Subsec:pf main}.

\begin{lemma}\label{prop:two}
    For every $k\ge1$, $R(K_{2,3k},\Z_3)=3k+3$.
\end{lemma}

\begin{lemma}\label{prop:three-six}
    $R(K_{3,6},\Z_3)=9$.
\end{lemma}

\begin{lemma}\label{prop:four-six}
    $R(K_{4,6},\Z_3)=10$.
\end{lemma}

\subsection{Lower-bound constructions for the exceptional $K_{s,3}$ values}\label{subsec:small-lower}

In this subsection, for $s\in\{3,4,5,7\}$, we show that $R(K_{s,3},\Z_3)\ge s+4$ by constructing an edge-labeling of $K_{s+3}$ without a zero-sum $K_{s,3}$ in \(\Z_3\).

For $s=3$, let the vertices of $K_6$ be $a,b,c,d,e,f$.  Assign weight $0$ to every edge incident with $a$. On the remaining five vertices, assign weight $1$ to the edges of the cycle $bcdefb$, and weight $2$ to the edges of its complementary cycle.
Then every vertex has weighted degree $0$.
Now consider any \(3\)-set \(X\) of vertices. If \(a\in X\), then the internal triangle sum of \(X\) is just the nonzero weight of the edge joining the other two vertices. If \(a\notin X\), then the three internal edge weights are of type \(1,1,2\) or \(1,2,2\), and hence again have nonzero sum in \(\mathbb Z_3\). Since all weighted degrees are 0, the cut sum of any \(3\)-set equals its internal triangle sum. Thus no \(3\)-set has zero cut sum, and so this edge-labeling of \(K_6\) contains no zero-sum copy of \(K_{3,3}\). Hence $R(K_{3,3},\Z_3)\ge7.$

For $s=4$, consider $K_7$ and fix a five-cycle $C$.  Assign weight $1$ to the edges of $C$, and weight $0$ to all other edges.  Every copy of \(K_{4,3}\) in \(K_7\) is determined by a cut. If the $4$-vertex side meets $C$ in $2$, $3$, or $4$ vertices, then the number of edges of $C$ crossing the cut is either $2$ or $4$, which is never $0$ in $\Z_3$. Hence this edge-labeling contains no zero-sum copy of \(K_{4,3}\), and therefore $R(K_{4,3},\Z_3)\ge8$.

For $s=5$, partition the vertices of $K_8$ into four pairs.  Assign weight $1$ to two of the pair-edges, weight $2$ to the other two pair-edges, and weight $0$ to every remaining edge.  A $3$-vertex side cuts an odd number of these four pair-edges, namely either one or three.  If it cuts one pair-edge, then the cut sum is $1$ or $2$.  If it cuts three pair-edges, then their weights are either $1,1,2$ or $1,2,2$, whose sums are $1$ and $2$ in $\Z_3$. Thus no \(3\)-cut has sum zero, so this edge-labeling contains no zero-sum copy of \(K_{5,3}\). Hence $R(K_{5,3},\Z_3)\ge9$.

For \(s=7\), partition the vertices of \(K_{10}\) into five pairs \(P_1,P_2,P_3,P_4,P_5.\)
Assign weight \(0\) to the edges inside the pairs. For \(i\ne j\), assign weight \(1\) to every edge between \(P_i\) and \(P_j\) if \(i-j\equiv \pm1 \pmod 5\), and weight \(2\) otherwise. Then each vertex has weighted degree
\(4\cdot 1+4\cdot 2=12\equiv 0 \pmod 3.\)
We claim that every \(3\)-set has nonzero internal triangle sum. Indeed, if two of its vertices lie in the same pair, then the three edge weights are \(0,c,c\) for some \(c\in\{1,2\}\), giving a nonzero sum in \(\mathbb Z_3\). If its vertices lie in three distinct pairs, then the three edge weights are of type \(1,1,2\) or \(1,2,2\), again with nonzero sum. Since all weighted degrees are $0$, the cut sum of any \(3\)-set equals its internal triangle sum. Hence no \(3\)-cut has sum $0$, and so this edge-labeling of \(K_{10}\) contains no zero-sum copy of \(K_{7,3}\). Therefore \(R(K_{7,3},\Z_3)\ge11.\)

\subsection{Proof of Theorem~\ref{thm:main} assuming Lemmas~\ref{prop:two}, \ref{prop:three-six}, and \ref{prop:four-six}}\label{Subsec:pf main}
In this subsection, we complete the proof of Theorem~\ref{thm:main} by determining the exact value of $R(K_{s,3k},\Z_3)$ according to the values of $s$ and $k.$

We shall use the following deletion step.  Let $K$ be a complete graph on vertex set $V$, let $w\colon E(K)\to\Z_3$ be an edge-labeling, and let $S\subseteq V$ with $|S|\ge11$.  Then there is a three-set $T\subseteq S$ such that $\partial(S\setminus T)=\partial(S)$.
To prove this, let $U=V\setminus S$.
For $x\in S$, define
\[
    b_x=\sum_{y\in S\setminus\{x\}}w(xy)-\sum_{u\in U}w(ux),
\]
and define an auxiliary edge-labeling $g$ on the pairs from $S$ by $g(xy)\coloneqq w(xy)+2b_x+2b_y$.
For every three-set $T\subseteq S$, Proposition~\ref{prop:cut-increment} gives
\[
    \partial(U\cup T)-\partial(U)
    =\sum_{t\in T}b_t-2\sum_{\{x,y\}\subseteq T}w(xy)
    =\sum_{t\in T}b_t+\sum_{\{x,y\}\subseteq T}w(xy),
\]
since $-2=1$ in $\Z_3$.  Also, each $b_t$ appears in exactly two edges of the triangle $T$, so
\[
    \sum_{\{x,y\}\subseteq T}g(xy)
    =\sum_{\{x,y\}\subseteq T}w(xy)+2\cdot2\sum_{t\in T}b_t
    =\sum_{\{x,y\}\subseteq T}w(xy)+\sum_{t\in T}b_t.
\]
Thus
\[
    \partial(U\cup T)-\partial(U)
    =\sum_{\{x,y\}\subseteq T}g(xy).
\]
Since $|S|\ge11$, Lemma~\ref{lem:cg} applied to the edge-labeling $g$ gives a three-set $T\subseteq S$ with
\[
    \sum_{\{x,y\}\subseteq T}g(xy)=0.
\]
For this $T$, we have $\partial(U\cup T)=\partial(U)$.  Since $U\cup T=V\setminus(S\setminus T)$ and complementary sets have the same cut sum, this is equivalent to $\partial(S\setminus T)=\partial(S)$.

First, when $s=2$ and $k\ge1$, Lemma~\ref{prop:two} gives $R(K_{2,3k},\Z_3)=3k+3=s+3k+1$.
Next, let $s\in\{3,4,5,7\}$ and $k=1$. Since $3\nmid 2$, Lemma~\ref{prop:two} together with Lemma~\ref{lem:side-lift} gives
\[R(K_{s,3},\Z_3)\le R(K_{2,3},\Z_3)+(s-2)=s+4=s+3k+1.\]
On the other hand, the constructions in Subsection~\ref{subsec:small-lower} show that $R(K_{s,3},\Z_3)\ge s+4$ for $s\in\{3,4,5,7\}$.
Therefore, $R(K_{s,3},\Z_3)=s+4=s+3k+1$ for $s\in\{3,4,5,7\}$ and $k=1$.

For $s=6$ and $k=1$, Lemma~\ref{prop:three-six} gives $R(K_{6,3},\Z_3)=R(K_{3,6},\Z_3)=9=s+3k$.

It remains to consider $s\ge8$ and $k=1.$  Let $K$ be a complete graph on vertex set $V$ with $|V|=s+3$, and let $w\colon E(K)\to\Z_3$ be an arbitrary edge-labeling. Set $S=V$.  Since $|S|=s+3\ge11$ and $\partial(S)=0$, the deletion step above gives a three-set $T\subseteq S$ such that $\partial(S\setminus T)=0$.  The set $S\setminus T$ has size $s$, and its complement has size $3$. Hence this cut gives a zero-sum copy of $K_{s,3}$.  Thus $R(K_{s,3},\Z_3)\le s+3$ for $s\ge8$, and equality follows from the vertex-count lower bound. Therefore,  $R(K_{s,3},\Z_3)=s+3=s+3k$ for $s\ge8$ and $k=1.$

We may now assume that $s\ge 3$ and $k\ge 2.$
For $s=3$ and $k=2$, Lemma~\ref{prop:three-six} gives $R(K_{3,6},\Z_3)=9=s+3k$.
Now suppose $s=3$ and $k\ge3$. Let $K$ be a complete graph on vertex set $V$ with $|V|=3k+3$, and let
$w\colon E(K)\to\Z_3$ be an arbitrary edge-labeling. Recall that $\td(x)=\sum_{y\ne x}w(xy)$. Define a new edge-labeling $a\colon E(K)\to\Z_3$ by $a(xy)=w(xy)-\td(x)-\td(y)$.  For every three-set $T=\{x,y,z\}$, we claim that
\[
    \sum_{\{u,v\}\subseteq T}a(uv)=\partial_w(T).
\]
Indeed, the left-hand side is $\te_w(T)-2\sum_{t\in T}\td(t)$, which equals $\te_w(T)+\sum_{t\in T}\td(t)$ in $\Z_3$. The right-hand side is $\sum_{t\in T}\td(t)-2\te_w(T)$, which is the same element of $\Z_3$.  If there were no zero-sum $K_{3,3k}$, then the edge-labeling $a$ would contain no zero-sum triangle, contradicting Lemma~\ref{lem:cg}.  Hence $R(K_{3,3k},\Z_3)\le3k+3$. The reverse inequality is the vertex-count lower bound, so $R(K_{3,3k},\Z_3)=3k+3=s+3k$ for $s=3$ and $k\ge3$.

Next, consider the case $s=4$ and $k\ge2$.  Let $K$ be a complete graph on vertex set $V$ with $|V|=3k+4$, and let $w\colon E(K)\to\Z_3$ be an arbitrary edge-labeling.  It suffices to find a four-set $A$ with $\partial_w(A)=0$.  Start with $S=V$.  Then $\partial(S)=0$ and $|S|\equiv1\pmod3$.  Repeatedly applying the deletion step above while $|S|>10$, we obtain a $10$-set $S$ such that $\partial(S)=0$.
Let $Q=V\setminus S$, and for each $s\in S$, define $h_s=\sum_{q\in Q}w(sq)$.  Since $\partial(S)=0$, we have $\sum_{s\in S}h_s=0$.  Define an edge-labeling $w'$ on the pairs from $S$ by $w'(uv)=w(uv)+2h_u+2h_v$.  For every $4$-set $A\subseteq S$, let $H_A=\sum_{a\in A}h_a$.  Since $|S|=10$, $|A|=4$, and $|S\setminus A|=6$, we have
\begin{align}
    \sum_{a\in A,\ b\in S\setminus A}w'(ab)
     & =\sum_{a\in A,\ b\in S\setminus A}w(ab)+2\left(6H_A+4\sum_{b\in S\setminus A}h_b\right)\notag \\
     & =\sum_{a\in A,\ b\in S\setminus A}w(ab)+2(-H_A)\notag                                         \\
     & =\sum_{a\in A,\ b\in S\setminus A}w(ab)+H_A=\partial_w(A).\label{eq:k10-transfer}
\end{align}
By Lemma~\ref{prop:four-six}, the edge-labeling $w'$ contains a $4$-set $A$ whose $K_{4,6}$-cut inside $S$ has sum $0$.  Equation~\eqref{eq:k10-transfer} then gives $\partial_w(A)=0$.  Hence $R(K_{4,3k},\Z_3)\le3k+4$ for $k\ge2$. The reverse inequality follows from the vertex-count lower bound, and therefore $R(K_{4,3k},\Z_3)=3k+4=s+3k$ for $s=4$ and $k\ge2.$

Finally, let $s\ge5$ and $k\ge2$. Since $3\nmid 4$, Lemma~\ref{lem:side-lift} gives
\[
    R(K_{s,3k},\Z_3)\le R(K_{4,3k},\Z_3)+(s-4)=3k+4+s-4=s+3k.
\]
Again, the reverse inequality is the vertex-count lower bound.  Thus $R(K_{s,3k},\Z_3)=s+3k$ for $s\ge4$ and $k\ge2$.
This completes the proof of Theorem~\ref{thm:main}.  \qed

\section{The proofs of Lemmas~\ref{prop:two}, \ref{prop:three-six}, and \ref{prop:four-six}}\label{Sec:pf lemmas}
In this section, we complete the proofs of Lemmas~\ref{prop:two}, \ref{prop:three-six}, and \ref{prop:four-six}, respectively.

\subsection{Proof of Lemma~\ref{prop:two}}\label{subsec:prop2}

We first prove the lower bound.  Let $K$ be a complete graph on vertex set $V$ with $|V|=3k+2$.  Choose a Hamilton cycle $C$ in $K$, assign weight $1$ to every edge of $C$, and assign weight $0$ to every other edge.  Any copy of $K_{2,3k}$ in $K$ uses all vertices.  If its two-vertex side is $A$, then its weight is exactly the number of edges of $C$ crossing the cut \((A,V\setminus A)\).  This number is $2$ if the two vertices of $A$ are adjacent on the cycle and $4$ otherwise.  In either case, it is nonzero in $\Z_3$.  Hence this edge-labeling contains no zero-sum copy of \(K_{2,3k}\), and therefore $R(K_{2,3k},\Z_3)>3k+2$. Thus $R(K_{2,3k},\Z_3)\ge3k+3$.

We now prove the upper bound.  Let $K$ be a complete graph on vertex set $V$ with $|V|=3k+3$, and let $w\colon E(K)\to\Z_3$ be an arbitrary edge-labeling.  For distinct vertices $x,y,z\in V$, define
\[
    p(x,y;z)=\sum_{u\in V\setminus\{x,y,z\}}\bigl(w(xu)+w(yu)\bigr).
\]
This is the weight of the copy of $K_{2,3k}$ with two-vertex side $\{x,y\}$ and whose \(3k\)-vertex side is $V\setminus\{x,y,z\}$.  Suppose, for the contrary, that
\begin{equation}\label{eq:k23-no-p}
    p(x,y;z)\ne0\text{ for all distinct }x,y,z\in V.
\end{equation}
For $v\in V$, recall the weighted degree $\td(v)=\sum_{u\ne v}w(vu)$.  Adding a fixed constant $c\in\Z_3$ to every edge does not change the weight of any copy of  $K_{2,3k}$, since $K_{2,3k}$ has $6k$ edges.  This operation changes each weighted degree by $(3k+2)c=2c$.  Thus we may normalize the edge-labeling so that, for a fixed vertex \(q\in V\), $\td(q)=0$.

Let $R=V\setminus\{q\}$.  For $y\in R$, set $a_y=w(qy)$, $b_y=\td(y)$ and $r_y=b_y-a_y$.
Since $\td(q)=0$, we have $\sum_{y\in R}a_y=0$.  Hence, we obtain that $B\coloneqq\sum_{y\in R}b_y=\sum_{y\in R}r_y$.

We next translate the assumption~\eqref{eq:k23-no-p} into structural constraints on the quantities \(b_y\) and \(r_y\).
For distinct $y,z\in R$, we have $p(q,y;z)\ne0$ and $p(q,z;y)\ne0$. Thus~\eqref{eq:k23-no-p} implies
\begin{equation}\label{eq:k23-two-forbidden}
    w(yz)\ne b_y+a_y-a_z\text{ and }w(yz)\ne b_z+a_z-a_y.
\end{equation}
The two forbidden values differ by $r_y-r_z$.  Therefore, if $r_y\ne r_z$, then the two forbidden values are distinct, so the conditions imply that \(w(yz)\) is the third element of \(\Z_3\), namely $w(yz)=-b_y-b_z$.  If $r_y=r_z$, then the two forbidden values coincide. Define $\delta_{yz}=w(yz)+b_y+b_z$.
\begin{equation}\label{eq:k23-delta-type}
    \delta_{yz}=0\text{ if }r_y\ne r_z\text{ and }\delta_{yz}\in\{1,2\}\text{ if }r_y=r_z.
\end{equation}
Indeed, in the case \(r_y=r_z\), the value \(\delta_{yz}=0\) is exactly the forbidden value from \eqref{eq:k23-two-forbidden}.

Now fix $y\in R$. Using \(w(yz)=\delta_{yz}-b_y-b_z\), the degree equation at \(y\) gives
\[
    b_y=a_y+\sum_{z\in R\setminus\{y\}}w(yz)=a_y+\sum_{z\in R\setminus\{y\}}\delta_{yz}-\sum_{z\in R\setminus\{y\}}b_y-\sum_{z\in R\setminus\{y\}}b_z.
\]
Since $|R|-1=3k+1\equiv1\pmod3$, this simplifies to
\begin{equation}\label{eq:k23-struct-degree}
    \sum_{z\in R\setminus\{y\}}\delta_{yz}=r_y+B.
\end{equation}
The condition $p(y,z;q)\ne0$ gives
\begin{equation}\label{eq:k23-struct-pair}
    \delta_{yz}-b_y-b_z+r_y+r_z\ne0.
\end{equation}
Finally, for distinct $y,z,t\in R$, the condition $p(y,z;t)\ne0$ gives
\begin{equation}\label{eq:k23-struct-triple}
    \delta_{yz}+b_y+b_z-b_t-\delta_{yt}-\delta_{zt}\ne0.
\end{equation}
Thus any counterexample to the upper bound yields the abstract structure described in the following lemma.

\begin{lemma}\label{lem:k23-structure}
    There is no finite set $X$ with data $r_x,b_x\in\Z_3$ and symmetric values $\delta_{xy}\in\Z_3$ for distinct $x,y\in X$ such that $|X|\ge 5$, $|X|\equiv2\pmod3$, $B\coloneqq\sum_xb_x=\sum_xr_x$, and conditions \eqref{eq:k23-delta-type}, \eqref{eq:k23-struct-degree}, \eqref{eq:k23-struct-pair}, and \eqref{eq:k23-struct-triple} all hold with $R$ replaced by $X$.
\end{lemma}

\begin{proof}
    For $i\in\Z_3$, let $X_i=\{x\in X:r_x=i\}$ and $C_i=\{b_x:x\in X_i\}$.

    First suppose that all three sets \(X_0,X_1,X_2\) are nonempty.
    If $i\ne j$, $x\in X_i$, and $y\in X_j$, then $\delta_{xy}=0$ by \eqref{eq:k23-delta-type}. Hence \eqref{eq:k23-struct-pair} gives
    \begin{equation}\label{eq:CiCj-avoid}
        b_x+b_y\ne i+j.
    \end{equation}
    Moreover, if \(x\in X_i\), \(y\in X_j\), and \(z\in X_\ell\), where \(i,j,\ell\) are distinct, then all three relevant \(\delta\)-terms are zero, and \eqref{eq:k23-struct-triple} gives
    \begin{equation}\label{eq:CiCjCl-avoid}
        b_z\ne b_x+b_y.
    \end{equation}
    These two restrictions imply $(C_0,C_1,C_2)=(\{0\},\{0\},\{1\})$ or $(C_0,C_1,C_2)=(\{0\},\{2\},\{0\})$.
    Indeed, if \(1\in C_0\), then \eqref{eq:CiCj-avoid} implies \(C_1\subseteq\{1,2\}\) and \(C_2\subseteq\{0,2\}\). Since \eqref{eq:CiCjCl-avoid} also forbids \(C_1+C_2\) from meeting \(C_0\), we obtain \(C_1=\{2\}\) and \(C_2=\{0\}\), which then violates \eqref{eq:CiCjCl-avoid}. Similarly, \(2\notin C_0\). Thus \(C_0=\{0\}\). Then \eqref{eq:CiCj-avoid} gives \(C_1\subseteq\{0,2\}\) and \(C_2\subseteq\{0,1\}\), while \eqref{eq:CiCjCl-avoid} implies \(C_1\cap C_2=\emptyset\) and forbids \(C_1+C_2\) from containing \(0\). The two possibilities displayed above are the only ones.

    Consider first
    \((C_0,C_1,C_2)=(\{0\},\{0\},\{1\}).\)
    For any internal edge of a layer \(X_i\), by~\eqref{eq:k23-delta-type}, the value of \(\delta\) is contained in \(\{1,2\}\). Applying \eqref{eq:k23-struct-triple} with two vertices in \(X_i\) and the third vertex in a suitable different layer implies that every internal \(\delta\)-edge in each \(X_i\) has value \(2\). Therefore \eqref{eq:k23-struct-degree} gives $2(|X_i|-1)=i+B$ for $i=0,1,2$. Here $B=|X_2|$ since the only nonzero \(b\)-values occur in \(X_2\).  These congruences imply $|X_0|\equiv0$, $|X_1|\equiv2$, and $|X_2|\equiv1$, so $|X|\equiv0$, contradicting $|X|\equiv2$.
    The second possibility,
    \((C_0,C_1,C_2)=(\{0\},\{2\},\{0\}),\)
    is analogous. In this case the same conditions imply that every internal \(\delta\)-edge has value \(1\). Thus $|X_i|-1=i+B$ for $i=0,1,2$. Since now \(B=2|X_1|\), these congruences imply $|X_0|\equiv0$, $|X_1|\equiv1$, and $|X_2|\equiv2$. Again \(|X|\equiv 0\), a contradiction. Hence at most two of the layers \(X_0,X_1,X_2\) are nonempty.

    Next suppose that exactly two layers are nonempty. If \(X_i\) and \(X_j\) are the two nonempty layers and \(u,v\in C_i\), then an internal \(\delta\)-edge in \(X_i\) joining vertices with \(b\)-values \(u\) and \(v\) must lie in
    \begin{equation}\label{eq:k23-D}
        \delta_i(u,v;C_j)
        =\{1,2\}
        \setminus (\{u+v-2i\}\cup \{c-u-v:c\in C_j\}).
    \end{equation}
    The first forbidden value comes from \eqref{eq:k23-struct-pair}, and the second set of forbidden values comes from \eqref{eq:k23-struct-triple} applied with the third vertex in the other layer $X_j$.

    We first claim that the two nonempty layers cannot be $X_1$ and $X_2.$
    Indeed, we first show that $|C_1|=|C_2|=1.$
    For cross-layer edges, \eqref{eq:CiCj-avoid} gives $u+v\ne0$ for all $u\in C_1$ and $v\in C_2$.
    Hence neither $C_1$ nor $C_2$ has three elements. Moreover, two two-element subsets of $\Z_3$ always contain a pair whose sum is $0$. Thus at least one of $C_1,C_2$ is a singleton. It remains to rule out the case in which the other has two elements.
    Suppose on the contrary that for some \(i\in\{1,2\}\), $|C_i|=2$ and $|C_j|=1,$ where $j=3-i.$
    Let $C_i=\{a,b\}$ and $C_j=\{c\}.$
    If $|X_j|\ge 2,$ then by~\eqref{eq:k23-D},
    the $\delta$-value of the edge of $X_j$ cannot be $2c-2j, a-2c, b-2c$. Since $a-2c\neq b-2c,$ one of $a-2c$ and $b-2c$ equals 0, i.e., $a=2c$ or $b=2c.$
    Then, by~\eqref{eq:CiCj-avoid}, $i+j\neq0,$ which is a contradiction to $j=3-i$.
    Therefore, $|X_j|=1,$ and thus $|X_i|\equiv1\pmod3$.
    Now $B=i|X_i|+j=i+j=0.$
    Choose the vertex of $X_j$, by~\eqref{eq:k23-struct-degree}, we get $0=j+B=j,$ which is a contradiction to $j=3-i$ and $i\in[2]$. Thus, $|C_1|=|C_2|=1.$

    Suppose that $0\in C_1,$ then by~\eqref{eq:CiCj-avoid}, $0\notin C_2$.
    If $C_1=\{0\}$ and $C_2=\{1\},$ then $B=|X_1|+2|X_2|=|X_2|,$ which implies that $|X|=|X_1|+|X_2|\equiv0\pmod3$, which contradicts the assumption.
    If $C_1=\{0\}$ and $C_2=\{2\},$ then $B=|X_1|+2|X_2|=2|X_2|,$ which implies that $|X_1|\equiv0, |X_2|\equiv2\pmod3$, since $|X|\equiv2\pmod3$.
    Choose two vertices from \(X_1\) and one vertex from $X_2$, by~\eqref{eq:k23-D}, we get a contradiction to the $\delta$ value of the edge in $X_1$.
    Therefore, we may assume that $0\notin C_1.$

    In the two cases \((C_1,C_2)=(\{1\},\{0\})\) and
    \((C_1,C_2)=(\{1\},\{1\})\), we have
    \(B\equiv2\), \(|X_1|\equiv2\), and \(|X_2|\equiv0\pmod3\).
    If \(C_2=\{0\}\), by~\eqref{eq:k23-D}, the $\delta$-value of edges of $X_1$ is 2, and if \(C_2=\{1\}\), then the $\delta$-value of edges of $X_1$ is 1.
    Choose a vertex from $X_1,$ by~\eqref{eq:k23-struct-degree}, if \(C_2=\{0\}\), then we get $2(|X_1|-1)=1+B=0,$
    if \(C_2=\{1\}\), then we get $|X_1|-1=1+B=0,$
    both of which contradict that $|X_1|\equiv2\pmod3$.
    Therefore, we may assume that $1\notin C_1.$

    For the case when $C_1=\{2\}$ and \(C_2=\{0\}\),
    we get that $B\equiv |X_1|+2|X_2|\equiv2|X_1|\pmod3$.
    Thus, $|X_1|\equiv2|X_2|\pmod3$, and hence $|X|\equiv0\pmod3$, contradicting $|X|\equiv2\pmod3$.
    If $C_1=\{2\}$ and \(C_2=\{2\}\), we get that $B\equiv |X_1|+2|X_2|\equiv2|X_1|+2|X_2|\pmod3$.
    Thus, $|X_1|\equiv0\pmod3$, $|X_2|\equiv2\pmod3$, and $B\equiv1\pmod3$.
    By~\eqref{eq:k23-D}, there is no allowable $\delta$-value for an internal edge in $X_1$. Since $X_1\ne\emptyset$, this gives $|X_1|=1$, contradicting $|X_1|\equiv0\pmod3$.
    Thus, the two nonempty layers can only be $X_0, X_1$ or $X_0, X_2.$

    Replacing \((r_x,b_x,\delta_{xy})\) by
    \((2r_x,2b_x,2\delta_{xy})\), if necessary, preserves all the conditions
    and interchanges the roles of \(X_1\) and \(X_2\). Hence it remains to
    consider the case \(X_0,X_1\ne\emptyset\).

    We first record a simple consequence of the pair and triple conditions.
    Suppose \(i\in\{0,1\}\), \(j=1-i\), and an internal edge in \(X_i\) joins
    vertices with \(b\)-values \(u,v\). If \(c\in C_j\), then, by
    \eqref{eq:k23-struct-pair} and by applying \eqref{eq:k23-struct-triple}
    with each of the three possible distinguished pairs, its \(\delta\)-value
    must lie in
    \[
        \Gamma_i(u,v;c)
        \coloneqq
        \{1,2\}\setminus
        \{u+v-2i,\ c-u-v,\ u+c-v,\ v+c-u\}.
    \]
    Thus \(\Gamma_i(u,v;c)\ne\emptyset\) whenever such an edge exists.

    For cross-layer edges, \eqref{eq:CiCj-avoid} gives $u+v\ne1$ for all $u\in C_0$ and $v\in C_1$.
    Hence neither \(C_0\) nor \(C_1\) has three elements; and if one of them
    has two elements, then the other is a singleton.
    First suppose that \(C_0=\{u,v\}\) and \(C_1=\{c\}\).
    The conditions \(u+c\ne1\) and \(v+c\ne1\) imply that
    \(c=2,0,1\) when \(C_0=\{0,1\},\{0,2\},\{1,2\}\), respectively.
    Substituting these three possibilities into \(\Gamma_0(u,v;c)\), only the first gives a nonempty allowable set.
    Thus the only possible case with \(|C_0|=2\) is
    \[
        (C_0,C_1)=(\{0,1\},\{2\}).
    \]
    If instead \(C_1=\{u,v\}\) and \(C_0=\{c\}\), the same check with \(\Gamma_1(u,v;c)\) gives an empty allowable set in all three cases.
    Therefore the only possible case in which one of \(C_0,C_1\) has two
    elements is $(C_0,C_1)=(\{0,1\},\{2\})$.

    It remains to consider the case where both \(C_0\) and \(C_1\) are
    singletons.  Write \(C_0=\{a\}\), \(C_1=\{b\}\), and let
    \(n_i=|X_i|\).  Then $an_0+bn_1\equiv B\equiv n_1$ and $n_0+n_1\equiv2\pmod3$.
    Together with \(a+b\ne1\), this leaves only
    \[
        (a,b)=(0,0),\ (0,2),\ (1,1),\ (2,1),
    \]
    since \((a,b)=(1,2)\) and \((2,0)\) would imply
    \(n_0+n_1\equiv0\pmod3\).

    The case \((a,b)=(1,1)\) is impossible: the two congruences give
    \(n_0\equiv0\) and \(n_1\equiv2\), so \(X_1\) contains an internal edge.
    However, we have $\Gamma_1(1,1;1)=\emptyset$, contradicting the existence of its \(\delta\)-value.

    Consequently the only remaining possibilities are
    \[
        (C_0,C_1)=(\{0\},\{0\}),\
        (\{0\},\{2\}),\
        (\{0,1\},\{2\}),\
        (\{2\},\{1\}).
    \]
    We now eliminate these four possibilities.
    First suppose
    \((C_0,C_1)=(\{0\},\{0\}).\)
    Then \(B=0\), and since \(B=\sum r_x\), we have \(|X_1|\equiv 0\). Hence \(|X_0|\equiv 2\). Inside \(X_0\), condition \eqref{eq:k23-struct-triple} implies that every triangle is monochromatic in the two possible \(\delta\)-values \(1\) and \(2\). Thus all internal \(\delta\)-edges of \(X_0\) have one common value \(c\in\{1,2\}\). The degree equation \eqref{eq:k23-struct-degree} at a vertex of \(X_0\) then gives
    \(c(|X_0|-1)=0,\)
    so \(|X_0|\equiv 1\), contradicting \(|X_0|\equiv 2\).
    Next suppose
    \((C_0,C_1)=(\{0\},\{2\}).\)
    Let \(n_i=|X_i|\). Then
    \(B=2n_1=n_1,\)
    so \(n_1\equiv 0\). The conditions imply that every internal \(\delta\)-edge of \(X_1\) has value \(1\). Therefore, by~\eqref{eq:k23-struct-degree}, a vertex of \(X_1\) gives
    \(n_1-1=1+B.\)
    Since \(B=0\), this gives \(n_1\equiv 2\), contradicting \(n_1\equiv 0\).
    Now suppose
    \((C_0,C_1)=(\{0,1\},\{2\}).\)
    Let
    \(n_{00}=|\{x\in X_0:b_x=0\}|\), \(n_{01}=|\{x\in X_0:b_x=1\}|\), \(n_1=|X_1|.\)
    The identity \(B=\sum r_x\) gives
    \(n_{01}+2n_1=n_1,\)
    so
    \(n_1\equiv 2n_{01}.\)
    Since \(|X|\equiv 2\), it follows that
    \(n_{00}\equiv 2.\)
    The conditions imply the following allowable internal \(\delta\)-values in \(X_0\):
    \[\delta=1 \text{ on } (0,0)\text{-pairs},\
        \delta=2 \text{ on } (0,1)\text{-pairs},\
        \delta=1 \text{ on } (1,1)\text{-pairs}.
    \]
    Thus the \(\delta\)-degree, inside \(X_0\), of a vertex with \(b\)-value \(0\) is
    \((n_{00}-1)+2n_{01}.\)
    By \eqref{eq:k23-struct-degree}, this must equal \(B=n_1\). Hence
    \((n_{00}-1)+2n_{01}=n_1.\)
    Using \(n_{00}\equiv 2\) and \(n_1\equiv 2n_{01}\), this becomes
    \(1+2n_{01}=2n_{01},\)
    a contradiction.
    Finally suppose
    \((C_0,C_1)=(\{2\},\{1\}).\)
    Let \(n_i=|X_i|\). Then
    \(B=2n_0+n_1=n_1,\)
    so \(n_0\equiv 0\). Since \(|X|\equiv 2\), we have \(n_1\equiv 2\). The conditions imply that every internal \(\delta\)-edge of \(X_0\) has value \(2\). Thus, by~\eqref{eq:k23-struct-degree}, a vertex of \(X_0\) gives
    \(2(n_0-1)=B=n_1,\) which yields \(1=2\), a contradiction. Hence exactly two layers cannot occur.

    It remains to consider the case in which all vertices lie in one layer.  Write $r_x=r$ for all $x\in X$.  Since $|X|\equiv2$, we have $B=\sum_xr_x=2r$.  Define $\beta_x=b_x+2r$.  Then
    \begin{equation}\label{eq:beta-sum-zero}
        \sum_x\beta_x=0.
    \end{equation}
    Condition \eqref{eq:k23-struct-pair} becomes $\delta_{xy}\ne\beta_x+\beta_y$.  Define $\eta_{xy}=\delta_{xy}+\beta_x+\beta_y$.  Then $\eta_{xy}=0$ unless $\beta_x+\beta_y=0$, and in the latter case $\eta_{xy}\in\{1,2\}$.  Equations \eqref{eq:k23-struct-degree} and \eqref{eq:k23-struct-triple} become
    \begin{equation}\label{eq:eta-degree}
        \sum_{y\ne x}\eta_{xy}=0\text{ for }x\in X,
    \end{equation}
    \begin{equation}\label{eq:eta-triple}
        \eta_{xy}-\eta_{xz}-\eta_{yz}+\beta_x+\beta_y+\beta_z+r\ne0
    \end{equation}
    for all distinct $x,y,z$.

    For \(i=0,1,2\), let $m_i=|\{x:\beta_x=i\}|$.  From $|X|\equiv2$ and \eqref{eq:beta-sum-zero}, we have
    \begin{equation}\label{eq:mi-basic}
        m_0+m_1+m_2\equiv2\text{ and }m_1+2m_2\equiv0.
    \end{equation}
    We first note that if \(m_0>0\), then
    \(m_0\equiv 1.\)
    Indeed, consider the vertices with \(\beta=0\). On this set, every \(\eta\)-edge has value \(1\) or \(2\), every vertex has \(\eta\)-degree \(0\) by \eqref{eq:eta-degree}, and \eqref{eq:eta-triple} becomes
    \(\eta_{xy}-\eta_{xz}-\eta_{yz}+r\ne 0.\)
    If \(r=0\), this implies that every triangle is monochromatic, so all edges have a common value \(c\in\{1,2\}\), and then \(c(m_0-1)=0\), giving \(m_0\equiv 1\). If \(r=1\), the allowed triangle patterns on the \(\beta=0\) vertices are exactly \((2,2,2)\) and \((1,1,2)\). Equivalently, the edges of value \(1\) form a cut. If this cut is nontrivial, say its parts have sizes $a$ and $m_0-a$ with $1\le a\le m_0-1$, then the \(\eta\)-degree of a vertex in the two parts gives $(m_0-a)+2(a-1)=0$ and $a+2(m_0-a-1)=0$, which are impossible in $\Z_3$. Hence all edges have value \(2\), and \(2(m_0-1)=0\), so \(m_0\equiv 1\). The case \(r=2\) is symmetric.

    Now consider the three possible values of \(r\).
    If \(r=0\), then three vertices with \(\beta=1\) would violate \eqref{eq:eta-triple}; hence \(m_1\le 2\). Similarly, \(m_2\le 2\). Together with \eqref{eq:mi-basic}, this leaves only
    \((m_1,m_2)=(0,0),\ (1,1),\ \text{or }(2,2).\)
    The case \((m_1,m_2)=(1,1)\) is impossible by \eqref{eq:eta-degree}, since the unique edge between the \(\beta=1\) vertex and the \(\beta=2\) vertex has nonzero \(\eta\)-value. If \((m_1,m_2)=(0,0)\), then \(m_0\equiv 2\), contradicting $m_0\equiv 1$. Thus \((m_1,m_2)=(2,2)\).
    Take two vertices \(u,v\) with \(\beta=1\) and one vertex \(z\) with \(\beta=2\). Applying \eqref{eq:eta-triple} with each of the three possible choices of the distinguished pair implies
    \(\eta_{uz}=\eta_{vz}=1.\)
    On the other hand, taking one vertex with \(\beta=1\) and two vertices with \(\beta=2\) implies that the corresponding \((1,2)\)-edges have \(\eta\)-value \(2\). This is impossible. Hence \(r\ne 0\).

    If \(r=1\), then one vertex with \(\beta=1\) and two vertices with \(\beta=2\) cannot occur. Indeed, if the two \((1,2)\)-edges have \(\eta\)-values $a,b\in\{1,2\}$ and the \((2,2)\)-edge has \(\eta=0\), then the three inequalities obtained from \eqref{eq:eta-triple} give $-a-b\ne0$, $a-b\ne0$ and $b-a\ne0$, which are impossible in $\Z_3$. Hence, if \(m_1>0\), then \(m_2\le 1\). If \(m_2=1\), then each vertex with \(\beta=1\) has exactly one possible nonzero incident \(\eta\)-edge, namely the edge to the unique \(\beta=2\) vertex, contradicting the zero-degree condition \eqref{eq:eta-degree}. Thus \(m_2=0\) whenever \(m_1>0\). Then, regardless of whether $m_1$ is equal to 0, \eqref{eq:mi-basic} gives \(m_1\equiv 0\), and hence \(m_0\equiv 2\), contradicting \(m_0\equiv 1\).
    The case \(r=2\) is symmetric to \(r=1\), obtained by multiplying all elements of \(\Z_3\) by \(2\).
    This final contradiction proves the lemma.
\end{proof}

We now return to the graph. The data \((R;r,b,\delta)\) constructed above satisfy all hypotheses of Lemma~\ref{lem:k23-structure}, which is impossible. Therefore \eqref{eq:k23-no-p} is false. Hence there exist distinct vertices \(x,y,z\in V\) such that $p(x,y;z)=0$.
By the definition of \(p(x,y;z)\), this gives a zero-sum copy of \(K_{2,3k}\) with two-vertex side \(\{x,y\}\) and \(3k\)-vertex side \(V\setminus\{x,y,z\}\). Thus
\(R(K_{2,3k},\Z_3)\le 3k+3.\)
Together with the lower bound, this proves
\[R(K_{2,3k},\Z_3)=3k+3,\]
which completes the proof of Lemma~\ref{prop:two}.

\subsection{Proof of Lemma~\ref{prop:three-six}}\label{subsec:prop36}

We first prove a counting lemma which will also be used in the next subsection.

\begin{lemma}\label{lem:cut-count}
    Let $K$ be a complete graph on vertex set $V$, let $w\colon E(K)\to\Z_3$ be an edge-labeling, and let $B,Q\subseteq V$ be disjoint.  Let $L=V\setminus(B\cup Q)$.  If $|L|\ge7$, then for every $\rho\in\Z_3$,
    \begin{equation}\label{eq:cut-count}
        \sum_{\substack{T\subseteq L,\ |T|\equiv\rho,\\ \partial(B\cup T)=0}}(-1)^{|T|}=0\ \text{in}\ \Z_3.
    \end{equation}
\end{lemma}

\begin{proof}
    Write $L=\{v_1,\ldots,v_\ell\}$, and encode a subset $T\subseteq L$ by a vector $x\in\{0,1\}^\ell$, where $x_i=1$ means $v_i\in T$.  The condition $|T|\equiv\rho$ is represented by the linear equation
    \[
        F_1(x)=x_1+\cdots+x_\ell-\rho=0.
    \]
    For \(u\in V\), let $\chi_u(x)$ be $1$ if $u\in B$, be $0$ if $u\in Q$, and be $x_i$ if $u=v_i\in L$.  Thus, on \(0\)-\(1\) vectors, \(\chi_u(x)\) is the indicator of the event \(u\in B\cup T\). Define
    \[
        F_2(x)=\sum_{\{u,v\}\subseteq V}w(uv)\bigl(\chi_u(x)+\chi_v(x)-2\chi_u(x)\chi_v(x)\bigr).
    \]
    For a \(0\)-\(1\) vector \(x\), the factor
    \(\chi_u(x)+\chi_v(x)-2\chi_u(x)\chi_v(x)\)
    is \(1\) exactly when precisely one of \(u,v\) is contained in \(B\cup T\), and is \(0\) otherwise. Hence
    \(F_2(x)=\partial(B\cup T)\)
    on \(0\)-\(1\) vectors. Moreover, \(\deg F_1\le 1\) and \(\deg F_2\le 2\).

    For \(a\in \Z_3\), the element \(1-a^2\) is equal to \(1\) if \(a=0\), and is equal to \(0\) otherwise. Hence
    \[
        \sum_{\substack{x\in\{0,1\}^\ell,\\ F_1(x)=F_2(x)=0}}
        (-1)^{x_1+\cdots+x_\ell}
        =
        \sum_{x\in\{0,1\}^\ell}
        (-1)^{x_1+\cdots+x_\ell}
        \bigl(1-F_1(x)^2\bigr)\bigl(1-F_2(x)^2\bigr).
    \]
    The polynomial
    \(\bigl(1-F_1^2\bigr)\bigl(1-F_2^2\bigr)\)
    has degree at most \(2\deg F_1+2\deg F_2\le 2(1+2)=6<\ell\). It therefore suffices to check the signed sum of a monomial of degree less than \(\ell\). Such a monomial omits at least one variable, say \(x_j\). Summing first over \(x_j\in\{0,1\}\) gives
    \(1-1=0.\)
    Thus every monomial contributes \(0\). Therefore, we have
    \[
        \sum_{\substack{x\in\{0,1\}^\ell,\\ F_1(x)=F_2(x)=0}}
        (-1)^{x_1+\cdots+x_\ell}=0.
    \]
    Translating back from \(x\) to \(T\subseteq L\), this is exactly
    \(
    \sum_{\substack{T\subseteq L,\ |T|\equiv \rho,\\ \partial(B\cup T)=0}}
    (-1)^{|T|}=0,
    \)
    proving \eqref{eq:cut-count}.
\end{proof}

The lower bound $R(K_{3,6},\Z_3)\ge9$ is the vertex-count lower bound.  For the upper bound, let $K$ be a complete graph on vertex set $V$ with $|V|=9$, and let $w\colon E(K)\to\Z_3$ be an arbitrary edge-labeling. Suppose that there is no zero-sum copy of $K_{3,6}$.  Since every copy of \(K_{3,6}\) in \(K\) is spanning, this assumption is equivalent to
\begin{equation}\label{eq:no-zero-3cut}
    \partial(A)\ne0\text{ for every }A\subseteq V\text{ with }|A|=3.
\end{equation}
Fix a vertex $v\in V$, and apply Lemma~\ref{lem:cut-count} with $B=\{v\}$, $Q=\emptyset$, $L=V\setminus\{v\}$, and $\rho=2$.  Since \(|L|=8\), every subset \(T\subseteq L\) satisfying \(|T|\equiv 2\pmod 3\) has size \(2\), \(5\), or \(8\).  If $|T|=2$, then $B\cup T$ is a three-set with zero cut sum, contradicting \eqref{eq:no-zero-3cut}.  If $|T|=5$, then $B\cup T$ has size $6$, and its complement has size $3$ and the same cut value, again contradicting \eqref{eq:no-zero-3cut}.  Thus the only possible solution is $T=L$.  This is indeed a solution, because $B\cup L=V$ and $\partial(V)=0$, and it contributes $(-1)^8=1$ to the signed sum.  This contradicts Lemma~\ref{lem:cut-count}, whose signed sum must be $0$.  Therefore, \(K\) contains a zero-sum copy of \(K_{3,6}\), and hence
\(R(K_{3,6},\Z_3)\le 9.\)
Combining the upper and lower bounds yields
\(R(K_{3,6},\Z_3)=9.\)

\subsection{Proof of Lemma~\ref{prop:four-six}}\label{subsec:prop46}

The lower bound $R(K_{4,6},\Z_3)\ge10$ is the vertex-count lower bound.  For the upper bound, let $K$ be a complete graph on vertex set $V$ with $|V|=10$, and let $w\colon E(K)\to\Z_3$ be an arbitrary edge-labeling. Suppose, for contradiction, that
\begin{equation}\label{eq:k10-no-fourcut}
    \partial(A)\ne0\text{ for every }A\subseteq V\text{ with }|A|=4.
\end{equation}
Thus $\partial(A)^2=1$ for every four-set $A$.

For $v\in V$, recall that $\td(v)=\partial(\{v\})=\sum_{u\ne v}w(uv)$.  We first establish the following identity
\begin{equation}\label{eq:k10-moment}
    \sum_{|A|=4}\partial(A)^2=2\sum_{v\in V}\td(v)^2.
\end{equation}
Expanding the left-hand side in the edge variables, a fixed edge crosses exactly
\(2\binom{8}{3}=112\equiv 1\)
four-cuts, so the coefficient of its square is \(1\). Two adjacent edges cross the same four-cut in
\(\binom{7}{3}+\binom{7}{2}=56\equiv 2\)
ways. Since the cross term in a square carries an additional factor of \(2\), the coefficient of their product is
\(2\cdot 56\equiv 1.\)
Two disjoint edges cross the same four-cut in
\(4\binom{6}{2}=60\equiv 0\)
ways, so their product has coefficient \(0\).
On the right-hand side of \eqref{eq:k10-moment}, the square of a fixed edge occurs in the degree squares of its two endpoints, and hence has coefficient
\(2\cdot 2\equiv 1.\)
The product of two adjacent edges occurs as a cross term in the degree square of their common endpoint, and thus also has coefficient
\(2\cdot 2\equiv 1.\)
Products of disjoint edges do not occur. The coefficients on the two sides therefore agree, proving \eqref{eq:k10-moment}.

By \eqref{eq:k10-no-fourcut},
\(
\sum_{|A|=4}\partial(A)^2=\binom{10}{4}=210=0.
\)
Equation~\eqref{eq:k10-moment} therefore gives $\sum_v \td(v)^2=0$.  For \(x\in\Z_3\), we have \(x^2=0\) if \(x=0\) and \(x^2=1\) otherwise. Consequently, the number of vertices of nonzero weighted degree is divisible by \(3\). Hence the number of vertices of weighted degree zero is congruent to
\(10\equiv 1\pmod 3.\)
In particular, at least one vertex has weighted degree zero. Fix such a vertex, call it \(q\), and set
\(R=V\setminus\{q\}.\)

We next show that every pair of vertices in $R$ is contained in a three-set with zero cut sum.  Fix a pair $P=\{x,y\}\subseteq R$. Apply Lemma~\ref{lem:cut-count} with $B=\{q\}$, $Q=P$, $L=V\setminus(B\cup Q)$, and $\rho=0$.  Here $|L|=7$, so the possible solution sizes are $0$, $3$, and $6$.  The empty set is a solution because $\partial(\{q\})=\td(q)=0$.  A solution with $|T|=3$ would make $B \cup T$ a four-set with zero cut sum, contradicting \eqref{eq:k10-no-fourcut}.  A subset $T$ of size $6$ has the form $T=L\setminus\{z\}$ for some $z\in R\setminus P$. In this case, $B\cup T=V\setminus(P\cup\{z\})$, and hence, since complementary sets determine the same cut, $\partial(P\cup\{z\})=0$.  Define
\[
    \lambda_P=\bigl|\{z\in R\setminus P:\partial(P\cup\{z\})=0\}\bigr|,
\]
The only solutions contributing to the signed sum in Lemma~\ref{lem:cut-count} are therefore the empty set and the \(\lambda_P\) solutions of size \(6\). Since both sizes are even, the lemma gives $1+\lambda_P=0$.  Hence
\begin{equation}\label{eq:k10-pair-extension}
    \lambda_P\equiv2\pmod3,
\end{equation}
and in particular $\lambda_P>0$.  Thus every pair in $R$ is contained in at least one three-set with zero cut sum.

Let
\(
Z=\{v\in V:\td(v)=0\}.
\)
As observed above,
\begin{equation}\label{eq:k10-zero-degree-count}
    |Z|\equiv1\pmod3.
\end{equation}
We now show that every zero three-cut contained in $R$ contains exactly two vertices of $Z$.  Let $B_0\subseteq R$ satisfy $|B_0|=3$ and $\partial(B_0)=0$.  Apply Lemma~\ref{lem:cut-count} with $B=B_0$, $Q=\emptyset$, $L=V\setminus B_0$, and $\rho=0$.  Again, \(|L|=7\), so the possible sizes of \(T\) are \(0\), \(3\), and \(6\).  The empty set is a solution because \(\partial(B_0)=0\).  A solution with $|T|=3$ would make $B_0\cup T$ a six-set with zero cut sum; its complement would then be a four-set with zero cut sum, contradicting \eqref{eq:k10-no-fourcut}.
Every six-set \(T\subseteq L\) has the form
\(T=L\setminus\{v\}\)
for some \(v\in V\setminus B_0\).
Then
\(B_0\cup T=V\setminus\{v\},\)
and therefore
\(\partial(B_0\cup T)=\partial(\{v\})=\td(v).\)
Thus the solutions of size \(6\) correspond exactly to the vertices in \(Z\setminus B_0\). Since both \(0\) and \(6\) are even, Lemma~\ref{lem:cut-count} gives $1+|Z\setminus B_0|=0$.
and hence $|Z\setminus B_0|\equiv2\pmod3$.  Together with \eqref{eq:k10-zero-degree-count}, this yields
\[
    |Z\cap B_0|=|Z|-|Z\setminus B_0|\equiv1-2\equiv2\pmod3.
\]
Since $|B_0|=3$, we conclude that
\begin{equation}\label{eq:k10-threecut-two-zero}
    |Z\cap B_0|=2.
\end{equation}

We can now derive the final contradiction. Suppose that \(R\) contains two vertices of nonzero weighted degree, then by \eqref{eq:k10-pair-extension} this pair is contained in a zero three-cut $B_0\subseteq R$, but such a $B_0$ would contain at most one zero-degree vertex, contradicting \eqref{eq:k10-threecut-two-zero}.  Hence $R$ contains at most one nonzero-degree vertex.
Since $q\in Z$, the whole graph contains at most one nonzero-degree vertex. Note that the number of nonzero-degree vertices is divisible by $3$, so in fact every vertex has degree zero, i.e. $Z=V$.
On the other hand, \eqref{eq:k10-pair-extension} guarantees the existence of a zero three-cut \(B_0\subseteq R\). Since \(Z=V\), this set satisfies
\(|Z\cap B_0|=3,\)
contradicting \eqref{eq:k10-threecut-two-zero}.
This contradiction shows that \eqref{eq:k10-no-fourcut} is impossible. Therefore every \(\Z_3\)-edge-labeling of \(K_{10}\) contains a zero-sum copy of \(K_{4,6}\). Hence
\(R(K_{4,6},\Z_3)=10.\)

\section*{Acknowledgements}

The authors would like to thank Professor Yair Caro for pointing out the elementary Erd\H{o}s--Ginzburg--Ziv upper bound and for helpful comments on related bipartite zero-sum Ramsey problems.


\begin{thebibliography}{99}

    \bibitem{AlonCaro1993}
    N. Alon and Y. Caro,
    \newblock On three zero-sum Ramsey-type problems,
    \newblock \emph{Journal of Graph Theory} 17 (1993), no. 2, 177--192.

    \bibitem{AlvaradoColucciParente2025}
    J. D. Alvarado, L. Colucci and R. Parente,
    \newblock On a problem of Caro on the \(\mathbb Z_3\)-Ramsey number of forests,
    \newblock arXiv:2503.01032, 2025.

    \bibitem{BD1990}
    A. Bialostocki and P. Dierker,
    \newblock Zero sum Ramsey theorems,
    \newblock \emph{Congressus Numerantium} \textbf{70} (1990), 119--130.

    \bibitem{BD1992}
    A. Bialostocki and P. Dierker,
    \newblock On the Erd\H{o}s--Ginzburg--Ziv theorem and the Ramsey numbers for stars and matchings,
    \newblock \emph{Discrete Mathematics} \textbf{110} (1992), 1--8.

    \bibitem{Caropersonal} 
    Y. Caro,
    \newblock Personal communications, July, 02, 2026.

    \bibitem{Caro1992Complete}
    Y. Caro,
    \newblock On zero-sum Ramsey numbers --- complete graphs,
    \newblock \emph{Quarterly Journal of Mathematics} \textbf{43} (1992),
    175--181.

    \bibitem{Caro1992Stars}
    Y. Caro,
    \newblock On zero-sum Ramsey numbers --- stars,
    \newblock \emph{Discrete Mathematics} \textbf{104} (1992), no. 1, 1--6.

        \bibitem{Caro1993Bipartite}
        Y. Caro,
        \newblock Zero-sum bipartite Ramsey numbers,
        \newblock \emph{Czechoslovak Mathematical Journal} 43 (1993), no. 1, 107--114.

    \bibitem{Caro1994}
    Y. Caro,
    \newblock A complete characterization of the zero-sum (mod 2) Ramsey numbers,
    \newblock \emph{Journal of Combinatorial Theory, Series A} \textbf{68} (1994), 205--211.

    \bibitem{Caro1996}
    Y. Caro,
    \newblock Zero-sum problems --- A survey,
    \newblock \emph{Discrete Mathematics} \textbf{152} (1996), 93--113.

    \bibitem{Caro1997}
    Y. Caro,
    \newblock Binomial coefficients and zero-sum Ramsey numbers,
    \newblock \emph{Journal of Combinatorial Theory, Series A} \textbf{80}
    (1997), 367--373.

    \bibitem{CaroMifsud2025}
    Y. Caro and X. Mifsud,
    \newblock On zero-sum Ramsey numbers modulo \(3\),
    \newblock arXiv:2502.03864, 2025.

    \bibitem{CaroRoditty1995}
    Y. Caro and Y. Roditty,
    \newblock A zero-sum conjecture for trees,
    \newblock \emph{Ars Combinatoria} \textbf{40} (1995), 89--96.

    \bibitem{Cauchy1813}
    A. Cauchy,
    \newblock Recherches sur les nombres,
    \newblock \emph{J. Ecole Polytech} \textbf{9} (1813), 99--116.

    \bibitem{ChiHeCyclesWheels2026}
    C. Chi and J. He,
    \newblock On zero-sum Ramsey numbers of cycles and wheels,
    \newblock arXiv:2605.14954, 2026.

    \bibitem{ChungGraham1983}
    F. R. K. Chung and R. L. Graham,
    \newblock Edge-colored complete graphs with precisely colored subgraphs,
    \newblock \emph{Combinatorica} 3 (1983), 315--324.

    \bibitem{ColucciDEmidio2025}
    L. Colucci and M. D'Emidio,
    \newblock A linear upper bound on the zero-sum Ramsey number of forests in \(\mathbb Z_p\),
    \newblock arXiv:2512.06229, 2025.

    \bibitem{Davenport1935}
    H. Davenport,
    \newblock On the addition of residue classes,
    \newblock \emph{Journal of the London Mathematical Society} \textbf{10} (1935), 30--32.

    \bibitem{EGZ1961}
    P. Erd\H{o}s, A. Ginzburg, and A. Ziv,
    \newblock Theorem in the additive number theory,
    \newblock \emph{Bulletin of the Research Council of Israel, Section F: Mathematics and Physics} \textbf{10F} (1961), 41--43.

    \bibitem{HarborthPiepmeyer1994}
    H. Harborth and L. Piepmeyer,
    \newblock The zero-sum Ramsey numbers \(r(K_4,\mathbb Z_3)\) and
    \(r(K_6,\mathbb Z_3)\),
    \newblock \emph{Congressus Numerantium} \textbf{101} (1994), 51--54.

    \bibitem{HarborthPiepmeyer1996}
    H. Harborth and L. Piepmeyer,
    \newblock Zero-sum Ramsey numbers modulo \(3\),
    \newblock \emph{Journal of Combinatorial Theory, Series A} \textbf{75}
    (1996), 145--147.

    \bibitem{HeathSimmons2026}
    E. Heath and A. Simmons,
    \newblock A linear upper bound on the \(\mathbb Z_p\)-Ramsey number of graphs with sufficiently large \(2\)-packing,
    \newblock arXiv:2605.21817, 2026.

    \bibitem{KatzLianMalekshahianShapiro2025}
    J. Katz, X. Lian, A. Malekshahian, and A. Shapiro,
    \newblock A linear upper bound for zero-sum Ramsey numbers of bounded degree graphs,
    \newblock arXiv:2512.17790, 2025.

    \bibitem{LetzterMorrison2024}
    S. Letzter and N. Morrison,
    \newblock Directed cycles with zero weight in \(\mathbb Z_p^k\),
    \newblock \emph{Journal of Combinatorial Theory, Series B} \textbf{168}
    (2024), 192--207.

    \bibitem{Ramsey1930}
    F. P. Ramsey,
    \newblock On a problem of formal logic,
    \newblock Proceedings of the London Mathematical Society, Series 2, \textbf{30} (1930), 264--286.

    \bibitem{Shapiro2026}
    A. Shapiro,
    \newblock A linear upper bound on zero-sum Ramsey numbers of \(d\)-degenerate graphs in \(\mathbb Z_p\),
    \newblock arXiv:2604.10864, 2026.

\end{thebibliography}
\end{document}